\documentclass[reqno,11pt]{article} %leqno is the option to put formula numbers on the left side
\usepackage[margin=1in]{geometry}  % set the margins to 1in on all sides
\usepackage{amsmath, amssymb}
\usepackage{amsthm} %theorem environment option
\allowdisplaybreaks
\newcommand{\Mod}[1]{\ (\textup{mod}\ #1)}
\theoremstyle{plain} %text of this environment is typesetted in italics
\newtheorem{theorem}{\indent\sc Theorem}[section]
\newtheorem{lemma}[theorem]{\indent\sc Lemma}

\newtheorem{proposition}[theorem]{\indent\sc Proposition}

\theoremstyle{definition} %text of this environment is typesetted in roman letters
\newtheorem{definition}[theorem]{\indent\sc Definition}

\newtheorem{remark}[theorem]{\indent\sc Remark}
\newtheorem{example}[theorem]{\indent\sc Example}

\makeatletter
\def\address#1#2{\begingroup
\noindent\parbox[t]{7.8cm}{%
\small{\scshape\ignorespaces#1}\par\vskip1ex
\noindent\small{\itshape E-mail address}%
\/: #2\par\vskip4ex}\hfill%
\endgroup}%
\makeatother
\title{Generators of Siegel modular function field of higher genus and level} %title of the paper
\author{
\textsc{Ja Kyung Koo, Dong Hwa Shin and Dong Sung Yoon$^*$} %names of authors
}

\date{} %leave empty
\begin{document}

\maketitle

%%%%%%%%%%%%%%% footnote %%%%%%%%%%%%%%%%
\footnote{ %2010 MSC numbers
2010 \textit{Mathematics Subject Classification}. Primary 11F46, Secondary 14K25.}
\footnote{ %key words and phrases
\textit{Key words and phrases}.
Siegel modular functions, Siegel families, theta constants.}
\footnote{
$^*$Corresponding author.\\
\thanks{The second named author was
supported by Hankuk University of Foreign Studies Research Fund of
2016.} }
%%%%%%%%%%%%%%%%%%%%%%%%%%%%%%%%%%%%%%%

\vspace{-1cm}
\begin{center}
Dedicated to the late professor Jun-ichi Igusa
\end{center}

\begin{abstract}
For positive integers $g$ and $N$, let $\mathcal{F}_N$ be the field
of meromorphic Siegel modular functions of genus $g$ and level $N$ whose Fourier coefficients belong to the $N$th cyclotomic field. We present explicit generators of $\mathcal{F}_N$ over $\mathcal{F}_1$
in terms of quotient of theta constants,
when $g\geq2$ and $N\geq 3$.
\end{abstract}

\section{Introduction}

Let $g$ and $N$ be positive integers. We denote by $\mathcal{F}_N$
the field of meromorphic Siegel modular functions of genus $g$ and level $N$
whose Fourier coefficients belong to the $N$th cyclotomic field ($\S$\ref{sect2}).
Then, $\mathcal{F}_1$ is an algebraic function field of transcendence degree
$g(g+1)/2$ (\cite{Klingen}), and $\mathcal{F}_N$ is a finite Galois extension of $\mathcal{F}_1$
with $\mathrm{Gal}(\mathcal{F}_N/\mathcal{F}_1)\simeq\mathrm{GSp}_{2g}(\mathbb{Z}/N\mathbb{Z})/\{\pm I_{2g}\}$
(\cite{Shimura00} or \cite{K-S-Y}).
\par
In particular, when $g=1$, it is well known that
\begin{equation*}
\mathcal{F}_N=\left\{\begin{array}{ll}
\mathbb{Q}(j(\tau)) & \textrm{if}~N=1,\\
\mathcal{F}_1\left(f_\mathbf{v}(\tau)~|~\textrm{$\mathbf{v}\in\mathbb{Q}^2$
has exact denominator $N$}\right)& \textrm{if}~N\geq2,
\end{array}\right.
\end{equation*}
where $j(\tau)$ is the elliptic modular function and $f_\mathbf{v}(\tau)$
are Fricke functions (\cite{Lang} and \cite{Shimura71}).
Moreover, if $K$ is an imaginary quadratic field and $\tau_0$ is a CM-point then the field
\begin{equation*}
K\mathcal{F}_N(\tau_0)=K\left(f(\tau_0)~|~\textrm{$f\in\mathcal{F}_N$ which is finite at $\tau_0$}\right)
\end{equation*}
is the ray class field of $K$ with conductor $N$ (\cite[Chapter 10, \S 1]{Lang}).
To extend the result to CM-field case, we first need to find generators of $\mathcal{F}_N$ for arbitrary genus $g$.
\par
When $g=2$, Igusa determined in \cite{Igusa62}  three generators of $\mathcal{F}_1$ in terms of Siegel Eisenstein series. Furthermore, when $g=3$, Tsuyumine gave in \cite{Tsuyumine} seven
explicit generators of $\mathcal{F}_1$ which are ratios of Siegel modular forms of weight
at most $30$. For higher genus and level, although Siegel proved in \cite{Siegel} that every function in $\mathcal{F}_N$ ($N\geq3$)
can be expressed as a ratio of theta constants, 
it seems to be hard to get a finite number of generators of $\mathcal{F}_N$ explicitly.
\par
Let $g\geq2$ and $N\geq2$. The purpose of this paper is to investigate explicit generators of $\mathcal{F}_N$ over $\mathcal{F}_1$.
To this end we shall first introduce the notion
of a primitive Siegel family (Definition \ref{defSiegelfamily}). By using the order formula
for one-variable Siegel functions (Proposition \ref{Siegelproperty} (v))
we shall also give a concrete example of a primitive Siegel family
whenever $N\neq2,4$ and $(2^g-1)\nmid N$
(Theorem \ref{Thetaprimitive}). Finally,
we shall present explicit generators of $\mathcal{F}_N$ over $\mathcal{F}_1$ as
quotient of theta constants developed in \cite{K-R-S-Y} when
$N\geq 3$ (Theorem \ref{mainresult}).

\section{Siegel modular functions}\label{sect2}

As a preliminary we shall briefly review meromorphic Siegel modular functions.
\par
Let $g$ be a positive integer. For a commutative ring $R$ with unity $1$,
we denote by
\begin{align*}
\mathrm{GSp}_{2g}(R)&=\left\{\alpha\in\mathrm{GL}_{2g}(R)~|~\alpha^TJ\alpha=\nu(\alpha)J~
\textrm{with}~\nu(\alpha)\in R^\times\right\},\\
\mathrm{Sp}_{2g}(R)&=\left\{\alpha\in\mathrm{GSp}_{2g}(R)~|~\nu(\alpha)=1\right\},
\end{align*}
where $J=\begin{bmatrix}O_g&-I_g\\I_g&O_g\end{bmatrix}$ and $\alpha^T$ stands for the transpose of $\alpha$.
Then, it is straightforward that $\alpha=\begin{bmatrix}A&B\\C&D\end{bmatrix}\in\mathrm{GL}_{2g}(R)$, where $A,B,C,D$ are $g\times g$ block matrices, belongs to $\mathrm{Sp}_{2g}(R)$ if and only if
\begin{equation*}
A^TD-C^TB=I_g,~A^TC=C^TA,~B^TD=D^TB.
\end{equation*}
Furthermore, we observe that if $\alpha\in\mathrm{GL}_{2g}(R)$ belongs to 
$\mathrm{Sp}_{2g}(R)$, then so does $\alpha^T$ (\cite[p. 17]{Shimura00}). 
\par
In particular, the group
\begin{equation*}
\mathrm{GSp}_{2g}(\mathbb{R})_+=\left\{\alpha\in\mathrm{GSp}_{2g}(\mathbb{R})~|~\nu(\alpha)>0\right\}
\end{equation*}
acts on the Siegel upper half-space
\begin{equation*}
\mathbb{H}_g=\left\{Z\in M_g(\mathbb{C})~|~Z^T=Z,~\mathrm{Im}(Z)~\textrm{is positive definite}\right\}
\end{equation*}
by
\begin{equation*}
\alpha(Z)=(AZ+B)(CZ+D)^{-1}\quad(\alpha=\begin{bmatrix}A&B\\C&D\end{bmatrix}\in\mathrm{GSp}_{2g}(\mathbb{R})_+,Z\in\mathbb{H}_g).
\end{equation*}
For a positive integer $N$, let
$\Gamma(N)$ be the principal congruence subgroup of level $N$
in $\mathrm{Sp}_{2g}(\mathbb{Z})$, namely
\begin{equation*}
\Gamma(N)=\left\{\alpha\in\mathrm{Sp}_{2g}(\mathbb{Z})~|~
\alpha\equiv I_{2g}\Mod{N\cdot M_{2g}(\mathbb{Z})}\right\}.
\end{equation*}
A holomorphic function $g:\mathbb{H}_g\rightarrow\mathbb{C}$ is called a
Siegel modular form of weight $k$ and level $N$ ($k\in\mathbb{Z}$) if it
satisfies
\begin{equation*}
g(\alpha(Z))=\det(CZ+D)^kg(Z)~\textrm{for all}~\alpha\in\Gamma(N).
\end{equation*}
When $g=1$, we further require that $g$ is holomorphic at every cusp. Then it can be written as
\begin{equation*}
g(Z)=\sum_{M}c(M)e\left(\frac{1}{N}\mathrm{tr}(MZ)\right)\quad(c(M)\in\mathbb{C}),
\end{equation*}
where $M$ runs over all $g\times g$ positive semi-definite symmetric matrices
over half integers with integral diagonal entries, and $e(z)=e^{2\pi iz}$ ($z\in\mathbb{C}$).
We call the above expression the Fourier expansion of $g$, and call $c(M)$ the Fourier coefficients of $g$.
\par
For a subfield $F$ of $\mathbb{C}$, we denote by
\begin{equation*}
\begin{array}{rcl}
\mathcal{M}_k(\Gamma(N),F)&=&\textrm{the space of Siegel modular forms of weight $k$ and level $N$}\\
&&\textrm{with Fourier coefficients in $F$},\vspace{0.1cm}\\
\mathcal{M}_k(F)&=&\displaystyle\bigcup_{N=1}^\infty\mathcal{M}_k(\Gamma(N),F),\vspace{0.1cm}\\
\mathcal{A}_0(\Gamma(N),F)&=&\textrm{the field of functions of the form $g/h$}\\
&&\textrm{with $g\in\mathcal{M}_k(F)$ and $h\in\mathcal{M}_k(F)\setminus\{0\}$ for the same weight $k$,}\\
&&\textrm{which are invariant under $\Gamma(N)$}.
\end{array}
\end{equation*}
In particular, we let
\begin{equation*}
\mathcal{F}_N=\mathcal{A}_0(\Gamma(N),\mathbb{Q}(\zeta_N)),~\textrm{where}~
\zeta_N=e(1/N).
\end{equation*}
As is well known, $\mathcal{F}_N$ is a Galois extension of $\mathcal{F}_1$
with
\begin{equation*}
\mathrm{Gal}(\mathcal{F}_N/\mathcal{F}_1)\simeq\mathrm{GSp}_{2g}(\mathbb{Z}/N\mathbb{Z})/\{\pm I_{2g}\}
\end{equation*}
(\cite[Theorem 8.10]{Shimura00} or \cite[Proposition 2.2]{K-S-Y}).
 More precisely, consider the decomposition
\begin{equation*}
\mathrm{GSp}_{2g}(\mathbb{Z}/N\mathbb{Z})/\{\pm I_{2g}\}\simeq
G_N\cdot\mathrm{Sp}_{2g}(\mathbb{Z}/N\mathbb{Z})/\{\pm I_{2g}\},
\end{equation*}
$\textrm{where}~G_N=
\left\{\begin{bmatrix}I_g&O_g\\O_g&\nu I_g\end{bmatrix}~|~\nu\in(\mathbb{Z}/N\mathbb{Z})^\times\right\}$.
The action of $\mathrm{GSp}_{2g}(\mathbb{Z}/N\mathbb{Z})/\{\pm I_{2g}\}$ on $\mathcal{F}_N$
is given as follows: Let $f=g/h$ be an element of $\mathcal{F}_N$ for some
$g,h\in\mathcal{M}_k(\mathbb{Q}(\zeta_N))$ with Fourier expansions
\begin{equation*}
g(Z)=\sum_M c(M)e\left(\frac{1}{N}\mathrm{tr}(MZ)\right)~\textrm{and}~
h(Z)=\sum_M d(M)e\left(\frac{1}{N}\mathrm{tr}(MZ)\right).
\end{equation*}
\begin{enumerate}
\item[(i)] An element $\begin{bmatrix}I_g&O_g\\O_g&\nu I_g\end{bmatrix}$ of $G_N$ acts on $f$ by
\begin{equation*}
f^{\left[\begin{smallmatrix}I_g&O_g\\O_g&\nu I_g\end{smallmatrix}\right]}
=\frac{\displaystyle\sum_M c(M)^\sigma e\left(\frac{1}{N}\mathrm{tr}(MZ)\right)}
{\displaystyle\sum_M d(M)^\sigma e\left(\frac{1}{N}\mathrm{tr}(MZ)\right)},
\end{equation*}
where $\sigma$ is the automorphism of $\mathbb{Q}(\zeta_N)$ defined by $\zeta_N^\sigma=\zeta_N^\nu$.
\item[(ii)] An element $\widetilde{\gamma}$ of $\mathrm{Sp}_{2g}(\mathbb{Z}/N\mathbb{Z})/\{\pm I_{2g}\}$
acts on $f$ by
\begin{equation*}
f^{\widetilde{\gamma}}=f\circ\gamma,
\end{equation*}
where ${\gamma}$ is any preimage of $\widetilde{\gamma}$ under the reduction
$\mathrm{Sp}_{2g}(\mathbb{Z})\rightarrow
\mathrm{Sp}_{2g}(\mathbb{Z}/N\mathbb{Z})/\{\pm I_2\}$.
\end{enumerate}

\section{Primitive Siegel families}

In this section we shall define a primitive Siegel family on which
$\mathrm{GSp}_{2g}(\mathbb{Z}/N\mathbb{Z})/\{\pm I_{2g}\}\simeq\mathrm{Gal}(\mathcal{F}_N/\mathcal{F}_1)$ acts in a natural way.
By making use of certain members of a primitive Siegel family, we shall construct Siegel modular function fields for various congruence subgroups of $\mathrm{Sp}_{2g}(\mathbb{Z}$).
\par
Let $g$ and $N$ be positive integers such that $N\geq2$. Let
\begin{equation*}
\mathcal{I}_N=\left\{\mathbf{v}\in\mathbb{Q}^{2g}~|~\textrm{$N$ is the
smallest positive integer so that $N\mathbf{v}\in\mathbb{Z}^{2g}$}\right\}.
\end{equation*}

\begin{definition}\label{defSiegelfamily}
We call a family
\begin{equation*}
\left\{f_\mathbf{v}(Z)\right\}_{\mathbf{v}\in\mathcal{I}_N}
\end{equation*}
a \textit{Siegel family} (indexed by $\mathcal{I}_N$) if
it satisfies the following properties:
\begin{enumerate}
\item[(S1)] Every $f_\mathbf{v}(Z)$ belongs to $\mathcal{F}_N$.
\item[(S2)] The group $\mathrm{GSp}_{2g}(\mathbb{Z}/N\mathbb{Z})/\{\pm I_{2g}\}$ ($\simeq\mathrm{Gal}(\mathcal{F}_N/\mathcal{F}_1)$) acts on the family as
\begin{equation*}
f_\mathbf{v}(Z)^\alpha
=f_{\alpha^T\mathbf{v}}(Z)\quad(\alpha\in\mathrm{GSp}_{2g}(\mathbb{Z}/N\mathbb{Z})/\{\pm I_{2g}\}).
\end{equation*}
\end{enumerate}
Moreover, we say that the family is \textit{primitive}
if
\begin{enumerate}
\item[(S3)] $f_\mathbf{v}=f_{\mathbf{v}'}$ if and only if
$\mathbf{v}\equiv\pm\mathbf{v}'\Mod{\mathbb{Z}^{2g}}$.
\end{enumerate}
\end{definition}

\begin{remark}
Note that if $\{f_\mathbf{v}(Z)\}_{\mathbf{v}\in\mathcal{I}_N}$ is a Siegel family, then
so is $\{f_\mathbf{v}(Z)^n\}_{\mathbf{v}\in\mathcal{I}_N}$ for every nonzero integer $n$.
\end{remark}

Let $\{\mathbf{e}_1,\mathbf{e}_2,\ldots,\mathbf{e}_{2g}\}$ be the
standard basis for $\mathbb{R}^{2g}$, and let
\begin{equation*}
\mathbf{e}=\mathbf{e}_1+\mathbf{e}_2+\cdots+\mathbf{e}_{2g}
\quad\textrm{and}\quad\mathbf{f}=\mathbf{e}_1+\cdots+\mathbf{e}_g.
\end{equation*}

\begin{proposition}\label{generators1}
If $\left\{f_\mathbf{v}(Z)\right\}_{\mathbf{v}\in\mathcal{I}_N}$ is
a primitive Siegel family, then
we have
\begin{equation}\label{F_N}
\mathcal{F}_N=
\mathcal{F}_1\left(
f_{(1/N)\mathbf{e}_1}(Z),
f_{(1/N)\mathbf{e}_2}(Z),
\ldots,
f_{(1/N)\mathbf{e}_{2g}}(Z),
f_{(1/N)\mathbf{e}}(Z)\right).
\end{equation}
\end{proposition}
\begin{proof}
Let $E$ be the field on the right-hand side of (\ref{F_N})
which is a subfield of $\mathcal{F}_N$ by (S1).
Suppose that an element $\alpha=\left[a_{ij}\right]$ of
$\mathrm{GSp}_{2g}(\mathbb{Z}/N\mathbb{Z})/\{\pm I_{2g}\}$
leaves $E$ fixed elementwise.
Let $\mathbf{r}_1,\mathbf{r}_2,\ldots,\mathbf{r}_{2g}$ be the row vectors of $\alpha$.
We then derive by (S2) 
\begin{equation*}
f_{(1/N)\mathbf{e}_j}(Z)
=f_{(1/N)\mathbf{e}_j}(Z)^\alpha
=f_{(1/N)\alpha^T\mathbf{e}_j}(Z)=
f_{(1/N)\mathbf{r}_j^T}(Z)\quad(j=1,2,\ldots,2g).
\end{equation*}
Thus we obtain $(1/N)\mathbf{e}_j
\equiv\pm(1/N)\mathbf{r}_j^T\Mod{\mathbb{Z}^{2g}}$ by (S3), and hence
\begin{equation*}
\alpha\equiv\begin{bmatrix}
\pm1 & 0 & \cdots & 0\\
0& \pm1  &\cdots & 0\\
\vdots & \vdots & \ddots & \vdots\\
0&0&\cdots& \pm1
\end{bmatrix}\Mod{N\cdot M_{2g}(\mathbb{Z})}.
\end{equation*}
And, since
\begin{equation*}
f_{(1/N)\mathbf{e}}(Z)=
f_{(1/N)\mathbf{e}}(Z)^\alpha=
f_{(1/N)\alpha^T\mathbf{e}}(Z)
\end{equation*}
by (S2), we get by (S3) 
\begin{equation*}
\begin{bmatrix}a_{11}\\a_{22}\\\vdots\\a_{2g\,2g}\end{bmatrix}\equiv
\pm\begin{bmatrix}1\\1\\\vdots\\1\end{bmatrix}
\Mod{N\cdot\mathbb{Z}^{2g}}.
\end{equation*}
Hence $\alpha$ represents the identity element of $\mathrm{Gal}(\mathcal{F}_N/\mathcal{F}_1)$, which proves $\mathcal{F}_N=E$, as desired.
\end{proof}

We further consider the following congruence subgroups of $\mathrm{Sp}_{2g}(\mathbb{Z})$:
\begin{align*}
\Gamma^1(N)&=\left\{\alpha\in\mathrm{Sp}_{2g}(\mathbb{Z})~|~
\alpha\equiv\begin{bmatrix}I_g & \mathrm{*} \\O_g & I_g\end{bmatrix}\Mod{N\cdot M_{2g}(\mathbb{Z})}\right\},\\
\Gamma_1(N)&=\left\{\alpha\in\mathrm{Sp}_{2g}(\mathbb{Z})~|~
\alpha\equiv\begin{bmatrix}I_g &O_g \\\mathrm{*} & I_g\end{bmatrix}\Mod{N\cdot M_{2g}(\mathbb{Z})}\right\}.
\end{align*}
Let $\mathcal{F}^1_N(\mathbb{Q})$
and $\mathcal{F}_{1,N}(\mathbb{Q})$ be the subfields of
$\mathcal{F}_N$ consisting of functions with rational Fourier coefficients
which are invariant under $\Gamma^1(N)$ and $\Gamma_1(N)$, respectively.

\begin{proposition}\label{generators2}
Let $\left\{f_\mathbf{v}(Z)\right\}_{\mathbf{v}\in\mathcal{I}_N}$ be
a primitive Siegel family. Then,
\begin{enumerate}
\item[\textup{(i)}] $\mathcal{F}^1_N(\mathbb{Q})=
\mathcal{F}_1\left(
f_{(1/N)\mathbf{e}_1}(Z),\ldots,
f_{(1/N)\mathbf{e}_g}(Z),
f_{(1/N)\mathbf{f}}(Z)\right)$.
\item[\textup{(ii)}]
$\mathcal{F}_{1,N}(\mathbb{Q})=
\mathbb{Q}\left(f(NZ),
f_{(1/N)\mathbf{e}_1}(NZ),\ldots,
f_{(1/N)\mathbf{e}_g}(NZ),
f_{(1/N)\mathbf{f}}(NZ)~|~f(Z)\in\mathcal{F}_1\right)$.
\end{enumerate}
\end{proposition}
\begin{proof}
\begin{enumerate}
\item[(i)] Let $L=\mathcal{F}_1\left(
f_{(1/N)\mathbf{e}_1}(Z),\ldots,
f_{(1/N)\mathbf{e}_g}(Z),
f_{(1/N)\mathbf{f}}(Z)\right)$. For any $\gamma\in\Gamma^1(N)$ and
$\mathbf{r}\in(1/N)\mathbb{Z}^g$, we see that
\begin{equation*}
\gamma^T\begin{bmatrix}\mathbf{r}\\\mathbf{0}\end{bmatrix}\equiv
\begin{bmatrix}I_g&\mathrm{*}\\O_g&I_g\end{bmatrix}\begin{bmatrix}\mathbf{r}\\\mathbf{0}\end{bmatrix}
\equiv\begin{bmatrix}\mathbf{r}\\\mathbf{0}\end{bmatrix}\Mod{\mathbb{Z}^{2g}}.
\end{equation*}
This implies that any function in $L$ is modular for $\Gamma^1(N)$.
Moreover, since
\begin{equation*}
\begin{bmatrix}I_g&O_g\\O_g&\nu I_g\end{bmatrix}\begin{bmatrix}\mathbf{r}\\\mathbf{0}\end{bmatrix}
\equiv\begin{bmatrix}\mathbf{r}\\\mathbf{0}\end{bmatrix}\Mod{\mathbb{Z}^{2g}}
\end{equation*}
for all $\nu\in(\mathbb{Z}/N\mathbb{Z})^\times$, every function in $L$ has rational Fourier coefficients. 
Thus we reach the inclusion $L\subseteq\mathcal{F}^1_N(\mathbb{Q})$.
\par
Now, let $\alpha=\left[a_{ij}\right]$ be an element of $\mathrm{Sp}_{2g}(\mathbb{Z}/N\mathbb{Z})/
\{\pm I_{2g}\}$ which leaves $L$ fixed. Let $\mathbf{r}_1,\mathbf{r}_2,\ldots,\mathbf{r}_{2g}$ be the row vectors of $\alpha$. Since
\begin{equation*}
f_{(1/N)\mathbf{e}_j}(Z)=f_{(1/N)\mathbf{e}_j}(Z)^\alpha
=f_{(1/N)\alpha^T\mathbf{e}_j}(Z)
\quad(j=1,\ldots, g)
\end{equation*}
by (S2), we attain by (S3) 
\begin{equation*}
\mathbf{e}_j\equiv\pm\alpha^T\mathbf{e}_j\equiv\pm\mathbf{r}_j
\Mod{N\cdot\mathbb{Z}^{2g}}\quad(j=1,\ldots, g).
\end{equation*}
This gives
\begin{equation*}
\alpha\equiv\left[\begin{array}{ccc|ccc}
\pm1 & \cdots & 0 & 0 & \cdots & 0\\
\vdots &\ddots & \vdots & \vdots& \ddots& \vdots\\
0&\cdots&\pm1&0&\cdots&0\\\hline
\mathrm{*} & \cdots & \mathrm{*} & \mathrm{*} & \cdots & \mathrm{*}\\
\vdots &\ddots & \vdots & \vdots& \ddots& \vdots\\
\mathrm{*}&\cdots&\mathrm{*}&\mathrm{*}&\cdots&\mathrm{*}
\end{array}\right]\Mod{N\cdot M_{2g}(\mathbb{Z})}.
\end{equation*}
And, since
\begin{equation*}
f_{(1/N)\mathbf{f}}(Z)=f_{(1/N)\mathbf{f}}(Z)^\alpha
=f_{(1/N)\alpha^T\mathbf{f}}(Z)
\end{equation*}
by (S2), we get by (S3) that
\begin{equation*}
\begin{bmatrix}
1\\\vdots\\1\\0\\\vdots\\0
\end{bmatrix}\equiv\pm\alpha^T\begin{bmatrix}
1\\\vdots\\1\\0\\\vdots\\0
\end{bmatrix}\equiv
\pm\left[\begin{array}{ccc|ccc}
\pm1 & \cdots & 0 & \mathrm{*} & \cdots & \mathrm{*}\\
\vdots &\ddots & \vdots & \vdots& \ddots& \vdots\\
0&\cdots&\pm1&\mathrm{*}&\cdots&\mathrm{*}\\\hline
0 & \cdots & 0 & \mathrm{*} & \cdots & \mathrm{*}\\
\vdots &\ddots & \vdots & \vdots& \ddots& \vdots\\
0&\cdots&0&\mathrm{*}&\cdots&\mathrm{*}
\end{array}\right]
\begin{bmatrix}
1\\\vdots\\1\\0\\\vdots\\0
\end{bmatrix}\equiv
\begin{bmatrix}
\pm1\\\vdots\\\pm1\\0\\\vdots\\0
\end{bmatrix}\Mod{N\cdot\mathbb{Z}^{2g}}.
\end{equation*}
Thus we achieve
\begin{equation*}
\alpha\equiv\pm\begin{bmatrix}I_g&O_g\\\mathrm{*}&\mathrm{*}\end{bmatrix}
\Mod{N\cdot M_{2g}(\mathbb{Z})}.
\end{equation*}
It then follows from the fact $\alpha\in\mathrm{Sp}_{2g}(\mathbb{Z}/N\mathbb{Z})/\{\pm I_{2g}\}$ that
\begin{equation*}
\alpha\equiv\pm\begin{bmatrix}I_g& O_g\\\mathrm{*}&I_g\end{bmatrix}
\Mod{N\cdot M_{2g}(\mathbb{Z})},
\end{equation*}
and so $\alpha$ leaves $\mathcal{F}^1_N(\mathbb{Q})$ fixed. This implies that
\begin{equation*}
\mathrm{Gal}(\mathcal{F}_N/L)\subseteq\mathrm{Gal}(\mathcal{F}_N/\mathcal{F}^1_N(\mathbb{Q})),
\end{equation*}
and hence we get the converse inclusion $L\supseteq\mathcal{F}^1_N(\mathbb{Q})$. Therefore we conclude that $L=\mathcal{F}^1_N(\mathbb{Q})$.
\item[(ii)] Let $R=\mathbb{Q}\left(f(NZ),
f_{(1/N)\mathbf{e}_1}(NZ),\ldots,
f_{(1/N)\mathbf{e}_g}(NZ),
f_{(1/N)\mathbf{f}}(NZ)~|~f(Z)\in\mathcal{F}_1\right)$.
One can readily check that in $\mathrm{Sp}_{2g}(\mathbb{R})$
\begin{equation*}
\Gamma^1(N)=\gamma\Gamma_1(N)\gamma^{-1},~\textrm{where}~
\gamma=\begin{bmatrix}\sqrt{N}I_g & O_g\\O_g&(1/\sqrt{N})I_g\end{bmatrix}.
\end{equation*}
Thus we have the isomorphism
\begin{eqnarray*}
\mathcal{F}^1_N(\mathbb{Q})&\rightarrow&\mathcal{F}_{1,N}(\mathbb{Q})\\
f(Z)&\mapsto&(f\circ\gamma)(Z)=f(NZ).
\end{eqnarray*}
This proves $R=\mathcal{F}_{1,N}(\mathbb{Q})$.
\end{enumerate}
\end{proof}

\section{Theta constants}

In this section, we shall give a concrete example of a Siegel family
in terms of ratios of theta constants. Furthermore, we shall develop
a useful lemma for later sections concerning primitivity
of a Siegel family. 
\par
Let $N\geq2$. For a vector
$\mathbf{v}=\begin{bmatrix}r\\s\end{bmatrix}\in(1/N)\mathbb{Z}^2\setminus\mathbb{Z}^2$,
the Siegel function $g_\mathbf{v}(\tau)$ is
defined on the complex upper half-plane
$\mathbb{H}_1=\left\{\tau\in\mathbb{C}~|~\mathrm{Im}(\tau)>0\right\}$ by
\begin{equation}\label{Siegel}
g_\mathbf{v}(\tau)=-q^{(1/2)
\mathbf{B}_2(r)}e(s(r-1)/2)(1-q^re(s))\prod_{n=1}^{\infty}
(1-q^{n+r}e(s))(1-q^{n-r}e(-s)),
\end{equation}
where $\mathbf{B}_2(r)=r^2-r+1/6$ is the second Bernoulli polynomial and $q=e(\tau)$.
It has no zeros nor poles on $\mathbb{H}_1$.
\par
For a real number $x$, let $\langle x\rangle$ be the
fractional part of $x$ in the interval $[0,1)$. 

\begin{proposition}\label{Siegelproperty}
We have the following properties of Siegel functions:
\begin{enumerate}
\item[\textup{(i)}] $g_\mathbf{v}(\tau)^{12N}$
depends only on $\pm\mathbf{v}\Mod{\mathbb{Z}^2}$.
\item[\textup{(ii)}] $g_\mathbf{v}(\tau)^{12N}$ belongs to $\mathcal{F}_N$.
\item[\textup{(iii)}] Each $\alpha\in\mathrm{GL}_2(\mathbb{Z}/N\mathbb{Z})/\{\pm I_2\}
\simeq\mathrm{Gal}(\mathcal{F}_N/\mathcal{F}_1)$ acts on the function $g_\mathbf{v}(\tau)^{12N}$ by
\begin{equation*}
\left(g_\mathbf{v}(\tau)^{12N}\right)^\alpha
=g_{\alpha^T\mathbf{v}}(\tau)^{12N}.
\end{equation*}
\item[\textup{(iv)}]
Let $n$ be a nonzero integer. Then $g_\mathbf{v}(\tau)^{12Nn}=g_{\mathbf{v}'}(\tau)^{12Nn}$ if and only if $\mathbf{v}\equiv\pm\mathbf{v}'\Mod{\mathbb{Z}^{2}}$.
\item[\textup{(v)}] $\mathrm{ord}_q~g_\mathbf{v}(\tau)=(1/2)\mathbf{B}_2(\langle r\rangle)$.
\end{enumerate}
\end{proposition}
\begin{proof}
See \cite[$\S$2.1]{K-L} and \cite[Example 3.1]{J-K-S}.
\end{proof}

\begin{remark}
So, for any nonzero integer $n$, $\left\{g_\mathbf{v}(\tau)^{12Nn}\right\}_{\mathbf{v}\in\mathcal{I}_N}$
is a Siegel family (when $g=1$) which is also called a Fricke family (\cite[p. 32--33]{K-L}).
Moreover, it is primitive.
\end{remark}

Now, let $g$ be a positive integer.
For $\mathbf{v}=\begin{bmatrix}v_1\\\vdots\\v_{2g}\end{bmatrix}\in\mathbb{R}^{2g}$, we denote by
\begin{equation*}
\mathbf{v}_u=\begin{bmatrix}v_1\\\vdots\\v_g\end{bmatrix},
\quad
\mathbf{v}_l=\begin{bmatrix}v_{g+1}\\\vdots\\v_{2g}\end{bmatrix},
\quad
\langle\mathbf{v}\rangle=
\begin{bmatrix}\langle v_1\rangle\\\vdots\\\langle v_{2g}\rangle\end{bmatrix}.
\end{equation*}
The theta constant $\theta_\mathbf{v}(Z)$
is defined by the following infinite series
\begin{equation*}
\theta_\mathbf{v}(Z)
=\sum_{\mathbf{n}\in\mathbb{Z}^g}
e\left(\frac{1}{2}(\mathbf{n}+\mathbf{v}_u)^TZ(\mathbf{n}+\mathbf{v}_u)
+(\mathbf{n}+\mathbf{v}_u)^T\mathbf{v}_l\right)\quad(Z\in\mathbb{H}_g).
\end{equation*}
Igusa (\cite[Theorem 2]{Igusa66}) showed that $\theta_\mathbf{v}(Z)$
is identically zero if and only if
$\langle\mathbf{v}\rangle$ belongs to the set
\begin{equation*}
S_-=\left\{\mathbf{a}\in\{0,1/2\}^{2g}~|~e(2\mathbf{a}_u^T\mathbf{a}_l)=-1\right\}.
\end{equation*}
Now, let
\begin{equation*}
S_+=\{0,1/2\}^{2g}\setminus S_-=
\left\{\mathbf{b}\in\{0,1/2\}^{2g}~|~e(2\mathbf{b}_u^T\mathbf{b}_l)=1\right\}.
\end{equation*}
One can then readily show that
\begin{equation}\label{order}
|S_-|=2^{g-1}(2^g-1)\quad\textrm{and}\quad
|S_+|=2^{g-1}(2^g+1).
\end{equation}
Recently, Koo et al. defined the function for $\mathbf{v}\in\mathcal{I}_N$
\begin{equation}\label{bigtheta}
\Theta_\mathbf{v}(Z)=
2^{4N}e\left(-2^gN(2^g-1)(2^g+1)\mathbf{v}_u^T\mathbf{v}_l\right)
\frac{\displaystyle \prod_{\mathbf{a}\in S_-}\theta_{\mathbf{a}-\mathbf{v}}(Z)^{4N(2^g+1)}}
{\displaystyle \prod_{\mathbf{b}\in S_+}
\theta_\mathbf{b}(Z)^{4N(2^g-1)}}\quad(Z\in\mathbb{H}_g)
\end{equation}
as a quotient of theta constants.

\begin{proposition}\label{bigthetaproperty}
We get the following properties of $\Theta_\mathbf{v}(Z)$:
\begin{enumerate}
\item[\textup{(i)}] $\Theta_\mathbf{v}(Z)$ depends only on $\pm\mathbf{v}\Mod{\mathbb{Z}^{2g}}$.
\item[\textup{(ii)}] It belongs to $\mathcal{F}_N$.
\item[\textup{(iii)}] For every $\alpha\in\mathrm{GSp}_{2g}(\mathbb{Z}/N\mathbb{Z})/\{\pm I_{2g}\}\simeq\mathrm{Gal}(\mathcal{F}_N/\mathcal{F}_1)$, it satisfies
$\Theta_\mathbf{v}(Z)^\alpha=\Theta_{\alpha^T\mathbf{v}}(Z)$.
\end{enumerate}
\end{proposition}
\begin{proof}
See \cite[Lemma 4.4 and Proposition 4.5]{K-R-S-Y}.
\end{proof}

\begin{remark}
\begin{enumerate}
\item[(i)] $\{\Theta_\mathbf{v}(Z)\}_{\mathbf{v}\in\mathcal{I}_N}$ becomes a Siegel family
satisfying (S1) and (S2).
\item[(ii)] If $g=1$, 
then one can obtain by using Jacobi's triple product identity that
\begin{equation*}
\Theta_{\mathbf{v}}(\tau)=
g_{\mathbf{v}}(\tau)^{12N}\quad(\tau\in\mathbb{H}_1)
\end{equation*}
(\cite[Remark 4.3]{K-R-S-Y}).
This shows that $\Theta_\mathbf{v}$ is a multivariable generalization of the Siegel function $g_\mathbf{v}$.
\item[(iii)] One can verify that $\Theta_\mathbf{v}(Z)$ becomes identically zero when $N=2$.
\end{enumerate}
\end{remark}

For $\tau_1,\ldots,\tau_g\in\mathbb{H}_1$, by $\mathrm{diag}(\tau_1,\ldots,\tau_g)$ we mean
the $g\times g$ diagonal matrix with diagonal entries $\tau_1,\ldots,\tau_g$, that is
\begin{equation*}
\mathrm{diag}(\tau_1,\ldots,\tau_g)=\begin{bmatrix}\tau_1 & \cdots & 0\\
\vdots & \ddots  & \vdots\\
0 & \cdots & \tau_g\end{bmatrix}.
\end{equation*}
Note that $\mathrm{diag}(\tau_1,\ldots,\tau_g)$ belongs to $\mathbb{H}_g$.

\begin{lemma}\label{diagproduct}
We have the relation
\begin{align*}
&\theta_\mathbf{v}\left(\mathrm{diag}(\tau_1,\ldots,\tau_g)\right)
\quad(\tau_1,\ldots,\tau_g\in\mathbb{H}_1)\\
&=\left\{\begin{array}{ll}
\displaystyle\prod_{k=1}^g\xi_k
g_{\left[\begin{smallmatrix}1/2-v_k\\1/2-v_{k+g}\end{smallmatrix}\right]}(\tau_k)
g_{\left[\begin{smallmatrix}1/2\\1/2\end{smallmatrix}\right]}(\tau_k)^{-1}
\theta_{\left[\begin{smallmatrix}0\\0\end{smallmatrix}\right]}(\tau_k)
& \textrm{if}~\begin{bmatrix}\langle v_k\rangle\\
\langle v_{k+g}\rangle\end{bmatrix}\neq\begin{bmatrix}1/2\\1/2\end{bmatrix}~
\textrm{for all}~k=1,\ldots,g,\\
0 & \textrm{otherwise},\\
\end{array}\right.
\end{align*}
where $\xi_k=e((2v_kv_{k+g}+v_k-v_{k+g})/4)$.
\end{lemma}
\begin{proof}
See \cite[Example 4.3]{E-K-S}.
\end{proof}

For each $k=1,\ldots,2g$, let
\begin{equation*}
n_{k,0}=\left|\{\mathbf{a}\in S_-~|~a_k=0\}\right|\quad\textrm{and}\quad
n_{k,1/2}=\left|\{\mathbf{a}\in S_-~|~a_k=1/2\}\right|.
\end{equation*}
One can readily see from (\ref{order}) that
\begin{equation}\label{nk0}
n_{k,0}=2^{2g-2}-2^{g-1}\quad\textrm{and}\quad
n_{k,1/2}=2^{2g-2}\quad(k=1,\ldots,2g).
\end{equation}
Since these values do not depend on $k$, we simply write
$n_0$ and $n_{1/2}$ in place of $n_{k,0}$ and $n_{k,1/2}$, respectively.

\begin{lemma}\label{orderlemma}
Let $g\geq2$ and $n$ be a nonzero integer, and let $\mathbf{v}=
\begin{bmatrix}v_1\\\vdots\\v_{2g}\end{bmatrix},
\mathbf{v}'=\begin{bmatrix}v_1'\\\vdots\\v_{2g}'\end{bmatrix}
\in\mathcal{I}_N$.
Assume that $\Theta_\mathbf{v}(Z)^n=\Theta_{\mathbf{v}'}(Z)^n$ and
$\begin{bmatrix}\langle v_k\rangle\\
\langle v_{k+g}\rangle\end{bmatrix}\not\in\{0,1/2\}^2$ for all  $k=1,\ldots,g$.
Then we attain
$\begin{bmatrix}\langle v_k'\rangle\\
\langle v_{k+g}'\rangle\end{bmatrix}\not\in\{0,1/2\}^2$
for all $k=1,\ldots,g$
and
\begin{equation}\label{orderformula}
n_0\mathbf{B}_2(\langle1/2+v_k\rangle)+n_{1/2}\mathbf{B}_2(\langle v_k\rangle)
=n_0\mathbf{B}_2(\langle1/2+v_k'\rangle)+n_{1/2}\mathbf{B}_2(\langle v_k'\rangle)\quad\textrm{for each}~k=1,\ldots,2g.
\end{equation}
\end{lemma}
\begin{proof}
Since $\Theta_\mathbf{v}(Z)^n=\Theta_{\mathbf{v}'}(Z)^n$, we get
\begin{equation*}
\Theta_\mathbf{v}(\mathrm{diag}(\tau_1,\ldots,\tau_g))^n=
\Theta_{\mathbf{v}'}(\mathrm{diag}(\tau_1,\ldots,\tau_g))^n\quad
(\tau_1,\ldots,\tau_g\in\mathbb{H}_1).
\end{equation*}
It then follows from the definition (\ref{bigtheta}) that
\begin{equation*}
\prod_{\mathbf{a}\in S_-}\theta_{\mathbf{a}-\mathbf{v}}(\mathrm{diag}(\tau_1,\ldots,\tau_g))
^{4N(2^g+1)n}\doteq
\prod_{\mathbf{a}\in S_-}\theta_{\mathbf{a}-\mathbf{v}'}(\mathrm{diag}(\tau_1,\ldots,\tau_g))
^{4N(2^g+1)n},
\end{equation*}
where $\doteq$ stands for the equality up to a root of unity.
Here we note that if $g\geq 2$ then for any $\mathbf{x}\in \{0,1/2\}^2$ and $k=1,\ldots,g$, there exists $\mathbf{a}=
\begin{bmatrix}a_1\\\vdots\\a_{2g}\end{bmatrix}\in S_-$ such that $\begin{bmatrix}\langle a_k\rangle\\
\langle a_{k+g}\rangle\end{bmatrix}=\mathbf{x}$.
Hence we know by Lemma \ref{diagproduct} that
\begin{equation*}
\begin{bmatrix}\langle v_k'\rangle\\
\langle v_{k+g}'\rangle\end{bmatrix}\not\in\{0,1/2\}^2 \textrm{~for all $k=1,\ldots,g$}
\end{equation*}
  and
\begin{equation*}
\prod_{\mathbf{a}\in S_-}\prod_{k=1}^g g_{\left[\begin{smallmatrix}1/2-a_k+v_k\\
1/2-a_{k+g}+v_{k+g}\end{smallmatrix}\right]}(\tau_k)^{4N(2^g+1)n}\doteq
\prod_{\mathbf{a}\in S_-}\prod_{k=1}^g g_{\left[\begin{smallmatrix}1/2-a_k+v_k'\\
1/2-a_{k+g}+v_{k+g}'\end{smallmatrix}\right]}(\tau_k)^{4N(2^g+1)n}.
\end{equation*}
Comparing the orders with respect to $e(\tau_k)$ by using Proposition \ref{Siegelproperty} (v),
we obtain 
\begin{equation}\label{order1}
\sum_{\mathbf{a}\in S_-}\mathbf{B}_2(\langle1/2-a_k+v_k\rangle)
=\sum_{\mathbf{a}\in S_-}\mathbf{B}_2(\langle1/2-a_k+v_k'\rangle)
\quad\textrm{for each}~k=1,\ldots,g.
\end{equation}
\par
On the other hand, acting $\begin{bmatrix}O_g&I_g\\-I_g&O_g\end{bmatrix}^T\in\mathrm{Sp}_{2g}(\mathbb{Z})$
    on both sides of $\Theta_\mathbf{v}(Z)^n=\Theta_{\mathbf{v}'}(Z)^n$, we
    have by Proposition \ref{bigthetaproperty} (iii) that
\begin{equation*}
\Theta_{\left[\begin{smallmatrix}
\mathbf{v}_l\\-\mathbf{v}_u
\end{smallmatrix}\right]}(Z)^n=
\Theta_{\left[\begin{smallmatrix}
\mathbf{v}_l'\\-\mathbf{v}_u'
\end{smallmatrix}\right]}(Z)^n.
\end{equation*}
In exactly the same way as the first part of this proof, one can also achieve
\begin{equation}\label{order2}
\sum_{\mathbf{a}\in S_-}\mathbf{B}_2(\langle1/2-a_k+v_{k+g}\rangle)
=\sum_{\mathbf{a}\in S_-}\mathbf{B}_2(\langle1/2-a_k+v_{k+g}'\rangle)\quad
\textrm{for each}~k=1,\ldots,g.
\end{equation}
Now, (\ref{order1}) and (\ref{order2}) yield the formula (\ref{orderformula}).
\end{proof}

\section {Primitivity of the Siegel family $\{\Theta_\mathbf{v}(Z)\}_{\mathbf{v}\in\mathcal{I}_N}$}

Assume that
\begin{equation}\label{mainassumption}
g\geq2,~N\neq1,2,4~\textrm{and}~(2^g-1)\nmid N.
\end{equation}
In this section we shall prove that the Siegel family 
$\{\Theta_\mathbf{v}(Z)^n\}_{\mathbf{v}\in\mathcal{I}_N}$ 
is primitive for every nonzero integer $n$. 

\begin{lemma}\label{mainlemma1}
With the assumption \textup{(\ref{mainassumption})}, let $n$ be any nonzero integer, and let $\mathbf{v},\mathbf{v}'\in\mathcal{I}_N$ such that
$\begin{bmatrix}\langle v_k\rangle\\
\langle v_{k+g}\rangle\end{bmatrix}\not\in\{0,1/2\}^2$ for all $k=1,\ldots,g$.
If $\Theta_\mathbf{v}(Z)^n=\Theta_{\mathbf{v}'}(Z)^n$, then we have $\mathbf{v}\equiv\pm\mathbf{v}'\Mod{\mathbb{Z}^{2g}}$.
\end{lemma}
\begin{proof}
We may assume by Proposition \ref{bigthetaproperty} (i) that
\begin{equation}\label{vk<1}
0\leq v_k,v_k'<1\quad(k=1,\ldots,2g).
\end{equation}
For simplicity, let
\begin{equation*}
V_k=Nv_k\quad\textrm{and}\quad V_k'=Nv_k'\quad(k=1,\ldots,2g).
\end{equation*}
First, we shall show that $V_k\equiv\pm V_k'\Mod{N}$ for each $k=1,\ldots,2g$
via the following four steps.
\begin{enumerate}
\item[(i)] Let $k$ be an index ($1\leq k\leq 2g$) such that
$0\leq v_k, v_k'<1/2$.
Since we are assuming $\Theta_\mathbf{v}(Z)^n=\Theta_{\mathbf{v}'}(Z)^n$,
we get by Lemma \ref{orderlemma} and (\ref{vk<1}) 
\begin{equation*}
n_0\mathbf{B}_2(1/2+v_k)+n_{1/2}\mathbf{B}_2(v_k)
=n_0\mathbf{B}_2(1/2+v_k')+n_{1/2}\mathbf{B}_2(v_k').
\end{equation*}
By multiplying both sides by $4N^2$ we establish
\begin{align*}
&n_0\left\{(N+2V_k)^2-2N(N+2V_k)+2N^2/3\right\}+n_{1/2}(4V_k^2-4NV_k+2N^2/3)\\
&=n_0\left\{(N+2V_k')^2-2N(N+2V_k')+2N^2/3\right\}+n_{1/2}(4V_k'^2-4NV_k+2N^2/3),
\end{align*}
from which we attain
\begin{equation*}
(V_k-V_k')\left\{(n_0+n_{1/2})(V_k+V_k')-n_{1/2}N\right\}=0.
\end{equation*}
If $V_k\neq V_k'$, then we achieve by (\ref{nk0}) 
\begin{equation*}
V_k+V_k'=\frac{n_{1/2}}{n_0+n_{1/2}}N=\frac{2^{g-1}}{2^g-1}N,
\end{equation*}
which is not an integer by the assumption $(2^g-1)\nmid N$. Thus we must have
$V_k=V_k'$.
\item[(ii)] Let $k$ be an index ($1\leq k\leq 2g$) such that
$0\leq v_k<1/2$ and $1/2\leq v_k'<1$. 
We have by Lemma \ref{orderlemma} and (\ref{vk<1}) 
\begin{equation*}
n_0\mathbf{B}_2(1/2+v_k)+n_{1/2}\mathbf{B}_2(v_k)
=n_0\mathbf{B}_2(-1/2+v_k')+n_{1/2}\mathbf{B}_2(v_k').
\end{equation*}
We then see that
\begin{equation*}
(V_k+V_k'-N)\left\{(n_0+n_{1/2})(V_k-V_k')+n_0N\right\}=0.
\end{equation*}
If $V_k\neq N-V_k'$, then we derive by (\ref{nk0}) 
\begin{equation*}
V_k-V_k'=-\frac{n_0}{n_0+n_{1/2}}N=-\frac{(2^{g-1}-1)}{2^g-1}N,
\end{equation*}
which is not an integer again by the assumption $(2^g-1)\nmid N$. 
Hence we should have
$V_k=N-V_k'$.
\item[(iii)] Let $k$ be an index ($1\leq k\leq 2g$) such that
$1/2\leq v_k<1$ and $0\leq v_k'<1/2$.
In a similar way to (ii), one can also show that
$V_k=N-V_k'$.
\item[(iv)] Let $k$ be an integer ($1\leq k\leq 2g$) such that
$1/2\leq v_k,v_k'<1$. 
We deduce by Lemma \ref{orderlemma} and (\ref{vk<1}) 
\begin{equation*}
n_0\mathbf{B}_2(-1/2+v_k)+n_{1/2}\mathbf{B}_2(v_k)
=n_0\mathbf{B}_2(-1/2+v_k')+n_{1/2}\mathbf{B}_2(v_k').
\end{equation*}
And, we get
\begin{equation*}
(V_k-V_k')\left\{(n_0+n_{1/2})(V_k+V_k')-(2n_0+n_{1/2})N\right\}=0.
\end{equation*}
If $V_k\neq V_k'$, then it follows from (\ref{nk0}) that
\begin{equation*}
V_k+V_k'=\frac{2n_0+n_{1/2}}{n_0+n_{1/2}}N=N+\frac{2^{g-1}-1}{2^g-1}N
\end{equation*}
which is not an integer by the assumption $(2^g-1)\nmid N$. 
Therefore we are forced to have $V_k=V_k'$.
\end{enumerate}
\par
Second, we shall justify that $V_k\equiv V_k'\Mod{N}$ for all $k=1,\ldots,2g$, or $V_k\equiv
-V_k'\Mod{N}$ for all $k=1,\ldots,2g$. 
Since the exact denominator of the vector $\mathbf{v}$ is $N$
which is $\geq3$ and $\neq4$ by the assumption, one can
take an index $j$ ($1\leq j\leq 2g$) such that
\begin{equation}\label{primden}
\textrm{the exact denominator of $v_j$ is $\geq3$ and $\neq 4$}.
\end{equation}
In particular, we have
\begin{equation}\label{Vjnot-Vj}
V_j\not\equiv -V_j\Mod{N}.
\end{equation}
Here, we may assume that
\begin{equation*}
1\leq j\leq g
\end{equation*}
because the action of $\begin{bmatrix}O_g&I_g\\-I_g&O_g\end{bmatrix}^T\in\mathrm{Sp}_{2g}(\mathbb{Z})$
    on both sides of $\Theta_\mathbf{v}(Z)^n=\Theta_{\mathbf{v}'}(Z)^n$ (if necessary) yields
\begin{equation*}
\Theta_{\left[\begin{smallmatrix}
\mathbf{v}_l\\-\mathbf{v}_u
\end{smallmatrix}\right]}(Z)^n=
\Theta_{\left[\begin{smallmatrix}
\mathbf{v}_l'\\-\mathbf{v}_u'
\end{smallmatrix}\right]}(Z)^n
\end{equation*}
by Proposition \ref{bigthetaproperty} (iii).
Let $i$ be an arbitrary index ($1\leq i\leq 2g$) such that $i\neq j$.
If $\langle v_i+v_j\rangle,\langle v_i-v_j\rangle\in\{0,1/2\}$,
then we get
$v_j\in\{0,1/4,1/2,3/4\}$, which contradicts (\ref{primden}).
So, we must have
\begin{equation*}
\langle v_i+v_j\rangle\neq0,1/2\quad\textrm{or}\quad
\langle v_i-v_j\rangle\neq0,1/2.
\end{equation*}
Now, there are four possible cases:
\begin{enumerate}
\item[(C1)] $1\leq i\leq g$ and $\langle v_i+v_j\rangle\neq0,1/2$.
\item[(C2)] $1\leq i\leq g$ and $\langle v_i-v_j\rangle\neq0,1/2$.
\item[(C3)] $g+1\leq i\leq 2g$ and $\langle v_i+v_j\rangle\neq0,1/2$.
\item[(C4)] $g+1\leq i\leq 2g$ and $\langle v_i-v_j\rangle\neq0,1/2$.
\end{enumerate}
For each $1\leq r,s\leq g$, let $E_{rs}$ be the $g\times g$ matrix
with $1$ at the entry $(r,s)$ and zeros everywhere else, and let
\begin{equation*}
E_{rs}'=\left\{
\begin{array}{ll}
E_{rs}+E_{sr}&\textrm{if $r\neq s$},\\
E_{rr}&\textrm{if $r=s$}.
\end{array}
\right.
\end{equation*}
Take
\begin{align*}
\alpha=\left\{\begin{array}{ll}
\begin{bmatrix}
I_g+E_{ij}&O_g\\
O_g&I_g-E_{ji}
\end{bmatrix}^T & \textrm{for (C1)},\vspace{0.1cm}\\
\begin{bmatrix}
I_g-E_{ij}&O_g\\
O_g&I_g+E_{ji}
\end{bmatrix}^T & \textrm{for (C2)},\vspace{0.1cm}\\
\begin{bmatrix}
I_g&O_g\\
E_{i-g\,j}'&I_g
\end{bmatrix}^T & \textrm{for (C3)},\vspace{0.1cm}\\
\begin{bmatrix}
I_g&O_g\\
-E_{i-g\,j}'&I_g
\end{bmatrix}^T & \textrm{for (C4)}.
\end{array}\right.
\end{align*}
which belongs to $\mathrm{Sp}_{2g}(\mathbb{Z})$.
Acting $\alpha$ on both sides of $\Theta_\mathbf{v}(Z)^n=\Theta_{\mathbf{v}'}(Z)^n$, we obtain
by Proposition \ref{bigthetaproperty} (iii) 
\begin{equation}\label{Thetau}
\Theta_\mathbf{u}(Z)^n=\Theta_{\mathbf{u}'}(Z)^n,
\end{equation}
where $\mathbf{u}=\begin{bmatrix}u_1\\\vdots\\u_{2g}\end{bmatrix}$,
$\mathbf{u}'=\begin{bmatrix}u_1'\\\vdots\\u_{2g}'\end{bmatrix}\in\mathcal{I}_N$ such that
\begin{align*}
(u_k,u_k')=
\left\{\begin{array}{lll}
\left\{\begin{array}{l}
(v_k,v_k')\\
(v_i+v_j,v_i'+v_j')\\
(-v_{i+g}+v_{j+g},-v_{i+g}'+v_{j+g}')
\end{array}\right.&
\begin{array}{l}
\textrm{if}~k\neq i,j+g,\\
\textrm{if}~k=i,\\
\textrm{if}~k=j+g
\end{array}
&\textrm{for (C1)},\vspace{0.1cm}\\
\left\{\begin{array}{l}
(v_k,v_k')\\
(v_i-v_j,v_i'-v_j')\\
(v_{i+g}+v_{j+g},v_{i+g}'+v_{j+g}')
\end{array}\right.&
\begin{array}{l}
\textrm{if}~k\neq i,j+g,\\
\textrm{if}~k=i,\\
\textrm{if}~k=j+g
\end{array}
&\textrm{for (C2)},\vspace{0.1cm}\\
\left\{\begin{array}{l}
(v_k,v_k')\\
(v_i+v_j,v_i'+v_j')\\
(v_{i-g}+v_{j+g},v_{i-g}'+v_{j+g}')
\end{array}\right.&
\begin{array}{l}
\textrm{if}~k\neq i,j+g\\
\textrm{if}~k=i\\
\textrm{if}~k=j+g
\end{array}&
\textrm{for (C3)},\vspace{0.1cm}\\
\left\{\begin{array}{l}
(v_k,v_k')\\
(v_i-v_j,v_i'-v_j')\\
(-v_{i-g}+v_{j+g},-v_{i-g}'+v_{j+g}')
\end{array}\right.&
\begin{array}{l}
\textrm{if}~k\neq i,j+g\\
\textrm{if}~k=i\\
\textrm{if}~k=j+g
\end{array}&
\textrm{for (C4)}.
\end{array}\right.
\end{align*}
Here, we observe by (\ref{primden}) that
\begin{equation*}
\begin{bmatrix}\langle u_k\rangle\\
\langle u_{k+g}\rangle\end{bmatrix}\not\in\{0,1/2\}^2
\quad\textrm{for all}~k=1,\ldots,g.
\end{equation*}
Thus we deduce by applying the first part of the proof to the equality (\ref{Thetau}) that
\begin{equation*}
\left\{\begin{array}{l}
V_i+V_j\equiv\pm(V_i'+V_j')\Mod{N}\quad\textrm{for (C1) and (C3)},\\
V_i-V_j\equiv\pm(V_i'-V_j')\Mod{N}\quad\textrm{for (C2) and (C4)}.
\end{array}\right.
\end{equation*}
Now, remember that
\begin{equation*}
V_i\equiv\pm V_i'\Mod{N}\quad\textrm{and}\quad
V_j\equiv\pm V_j'\Mod{N}
\end{equation*}
derived in the first part of the proof.
For simplicity, let (A) represent the cases (C1) and (C3),
and let (B) represent the cases (C2) and (C4).
We then have eight possibilities for each case of (A) and (B):
\begin{enumerate}
\item[(A1)] $V_i+V_j\equiv V_i'+V_j'\Mod{N}$, $V_i\equiv V_i'\Mod{N}$, $V_j\equiv V_j'\Mod{N}$.
\item[(A2)] $V_i+V_j\equiv V_i'+V_j'\Mod{N}$, $V_i\equiv V_i'\Mod{N}$, $V_j\equiv-V_j'\Mod{N}$.
\item[(A3)] $V_i+V_j\equiv V_i'+V_j'\Mod{N}$, $V_i\equiv-V_i'\Mod{N}$, $V_j\equiv V_j'\Mod{N}$.
\item[(A4)] $V_i+V_j\equiv V_i'+V_j'\Mod{N}$, $V_i\equiv-V_i'\Mod{N},~V_j\equiv-V_j'\Mod{N}$.
\item[(A5)] $V_i+V_j\equiv-(V_i'+V_j')\Mod{N}$, $V_i\equiv V_i'\Mod{N},~V_j\equiv V_j'\Mod{N}$.
\item[(A6)] $V_i+V_j\equiv-(V_i'+V_j')\Mod{N}$, $V_i\equiv V_i'\Mod{N},~V_j\equiv-V_j'\Mod{N}$.
\item[(A7)] $V_i+V_j\equiv-(V_i'+V_j')\Mod{N}$, $V_i\equiv-V_i'\Mod{N},~V_j\equiv V_j'\Mod{N}$.
\item[(A8)] $V_i+V_j\equiv-(V_i'+V_j')\Mod{N}$, $V_i\equiv-V_i'\Mod{N},~V_j\equiv-V_j'\Mod{N}$.
\item[(B1)] $V_i-V_j\equiv V_i'-V_j'\Mod{N}$, $V_i\equiv V_i'\Mod{N},~V_j\equiv V_j'\Mod{N}$.
\item[(B2)] $V_i-V_j\equiv V_i'-V_j'\Mod{N}$, $V_i\equiv V_i'\Mod{N},~V_j\equiv-V_j'\Mod{N}$.
\item[(B3)] $V_i-V_j\equiv V_i'-V_j'\Mod{N}$, $V_i\equiv-V_i'\Mod{N},~V_j\equiv V_j'\Mod{N}$.
\item[(B4)] $V_i-V_j\equiv V_i'-V_j'\Mod{N}$, $V_i\equiv-V_i'\Mod{N},~V_j\equiv-V_j'\Mod{N}$.
\item[(B5)] $V_i-V_j\equiv-(V_i'-V_j')\Mod{N}$, $V_i\equiv V_i'\Mod{N},~V_j\equiv V_j'\Mod{N}$.
\item[(B6)] $V_i-V_j\equiv-(V_i'-V_j')\Mod{N}$, $V_i\equiv V_i'\Mod{N},~V_j\equiv-V_j'\Mod{N}$.
\item[(B7)] $V_i-V_j\equiv-(V_i'-V_j')\Mod{N}$, $V_i\equiv-V_i'\Mod{N},~V_j\equiv V_j'\Mod{N}$.
\item[(B8)] $V_i-V_j\equiv-(V_i'-V_j')\Mod{N}$, $V_i\equiv-V_i'\Mod{N},~V_j\equiv-V_j'\Mod{N}$.
\end{enumerate}
One can then readily check that
(A2), (A7), (B2) and (B7) contradict (\ref{Vjnot-Vj}). Furthermore,
(A4) contradicts $\langle v_i+v_j\rangle\neq0,1/2$, and (B4) contradicts
$\langle v_i-v_j\rangle\neq0,1/2$. Other ten cases yield
\begin{equation*}
\left\{\begin{array}{l}
V_i\equiv V_i'\Mod{N}~\textrm{and}~V_j\equiv V_j'\Mod{N},\quad\textrm{or}\\
V_i\equiv-V_i'\Mod{N}~\textrm{and}~V_j\equiv-V_j'\Mod{N}.\end{array}\right.
\end{equation*}
Therefore, we claim by (\ref{Vjnot-Vj}) that
\begin{equation*}
\left\{\begin{array}{l}
V_k\equiv V_k'\Mod{N}~\textrm{for all}~k=1,\ldots,2g,\quad\textrm{or}\\
V_k\equiv-V_k'\Mod{N}~\textrm{for all}~k=1,\ldots,2g,
\end{array}\right.
\end{equation*}
and hence $\mathbf{v}\equiv\pm\mathbf{v}'\Mod{\mathbb{Z}^{2g}}$, as desired.
\end{proof}

\begin{lemma}\label{mainlemma2}
With the assumption \textup{(\ref{mainassumption})}, let
 $n$ be any nonzero integer, and let $\mathbf{v},\mathbf{v}'\in\mathcal{I}_N$ such that
$\begin{bmatrix}\langle v_{k_i}\rangle\\
\langle v_{k_i+g}\rangle\end{bmatrix}\in\{0,1/2\}^2$ for some $1\leq k_1,\ldots,k_m\leq g$.
If $\Theta_\mathbf{v}(Z)^n=\Theta_{\mathbf{v}'}(Z)^n$, then we have $\mathbf{v}\equiv\pm\mathbf{v}'\Mod{\mathbb{Z}^{2g}}$.
\end{lemma}
\begin{proof}
As in Lemma \ref{mainlemma1} we take an index $j$ ($1\leq j\leq g$) such that
the exact denominator of $v_j$ is $\geq3$ and $\neq 4$.
Then we derive
\begin{equation*}
\langle v_{k_i}+v_j\rangle\neq0,1/2\quad\textrm{for all}~i=1,\ldots,m.
\end{equation*}
Let
\begin{equation*}
\beta=
\begin{bmatrix}
I_g+\sum_{i=1}^m E_{k_ij} & O_g\\
O_g & I_g-\sum_{i=1}^m E_{jk_i}
\end{bmatrix}^T
\end{equation*}
which belongs to $\mathrm{Sp}_{2g}(\mathbb{Z})$.
Acting $\beta$ on both sides of $\Theta_\mathbf{v}(Z)^n=\Theta_{\mathbf{v}'}(Z)^n$,
we get
\begin{equation}\label{ww'}
\Theta_\mathbf{w}(Z)^n=\Theta_{\mathbf{w}'}(Z)^n,
\end{equation}
where $\mathbf{w}=\begin{bmatrix}w_1\\\vdots\\w_{2g}\end{bmatrix}$,
$\mathbf{w}'=\begin{bmatrix}w_1'\\\vdots\\w_{2g}'\end{bmatrix}\in\mathcal{I}_N$ such that
\begin{equation*}
(w_k,w_k')
=\left\{\begin{array}{ll}
(v_k,v_k') & \textrm{if}~k\neq k_1,\ldots,k_m,j+g,\\
(v_{k_i}+v_j,v_{k_i}'+v_j') & \textrm{if}~k=k_i~(i=1,\ldots,m),\\
(v_{j+g}-\sum_{i=1}^mv_{k_i+g},v_{j+g}'-\sum_{i=1}^mv_{k_i+g}') & \textrm{if}~k=j+g.
\end{array}\right.
\end{equation*}
Note that
\begin{equation*}
\begin{bmatrix}
\langle w_k\rangle\\
\langle w_{k+g}\rangle\end{bmatrix}\not\in\{0,1/2\}^2
\quad\textrm{for all}~k=1,\ldots,g.
\end{equation*}
Then we obtain by applying Lemma \ref{mainlemma1} to the equality (\ref{ww'}) that
$\mathbf{w}\equiv\pm\mathbf{w}'\Mod{\mathbb{Z}^{2g}}$,
and hence we conclude that $\mathbf{v}\equiv\pm\mathbf{v}'\Mod{\mathbb{Z}^{2g}}$.
\end{proof}

By Proposition \ref{bigthetaproperty}, Lemmas \ref{mainlemma1} and \ref{mainlemma2}
we establish the following result as mentioned in the beginning of this section. 

\begin{theorem}\label{Thetaprimitive}
Assume that $g\geq2$, $N\neq1,2,4$ and $(2^g-1)\nmid N$. Then,
the Siegel family $\{\Theta_\mathbf{v}(Z)^n\}_{\mathbf{v}\in\mathcal{I}_N}$ is primitive for every nonzero integer $n$.
\end{theorem}

\section {Explicit generators of Siegel modular function fields}

By improving Lemmas \ref{mainlemma1} and \ref{mainlemma2} in some special cases
we shall obtain our main result on generators of Siegel modular function fields 
for various congruence subgroups. 

\begin{lemma}\label{weaklemma}
Let $n$ be any nonzero integer and $\mathbf{v}\in\mathcal{I}_N$. Assume that
\begin{equation}\label{weakassumption}
g\geq 2~\textrm{ and }~N\geq 3.
\end{equation}
\begin{enumerate}
\item[\textup{(i)}] If $\Theta_\mathbf{v}(Z)^n=\Theta_{(1/N)\mathbf{f}}(Z)^n$, then we have
$\mathbf{v}\equiv\pm(1/N)\mathbf{f}\Mod{\mathbb{Z}^{2g}}$.
\item[\textup{(ii)}] If $\Theta_\mathbf{v}(Z)^n=\Theta_{(1/N)\mathbf{e}}(Z)^n$, then we have
$\mathbf{v}\equiv\pm(1/N)\mathbf{e}\Mod{\mathbb{Z}^{2g}}$.
\item[\textup{(iii)}] If $\Theta_\mathbf{v}(Z)^n=\Theta_{(1/N)\mathbf{e}_j}(Z)^n$ for $1\leq j\leq 2g$, then we have $\mathbf{v}\equiv\pm(1/N)\mathbf{e}_j\Mod{\mathbb{Z}^{2g}}$.
\end{enumerate}
\end{lemma}
\begin{proof}
\begin{enumerate}
\item[(i)] 
By Proposition \ref{bigthetaproperty} (i) we may assume that
$0\leq v_k<1$ for all $k=1,\ldots,2g$.
We shall first show that
\begin{equation*}
v_k\equiv\left\{\begin{array}{ll}
\pm {1}/{N}\Mod{\mathbb{Z}}&\textrm{if $1\leq k\leq g$ and $N\neq 2^g-1$},\\
\pm {1}/{N}~\textrm{or}~{1}/{2}\pm{1}/{2N}\Mod{\mathbb{Z}}&\textrm{if $1\leq k\leq g$ and $N= 2^g-1$},\\
0\Mod{\mathbb{Z}} & \textrm{if}~g+1\leq k\leq 2g
\end{array}\right.
\end{equation*}
via the following four steps.
\begin{enumerate}
\item[(a)] Let $k$ be an index ($1\leq k\leq g$) such that $0\leq v_k<1/2$.
Since we are assuming $\Theta_\mathbf{v}(Z)^n=\Theta_{(1/N)\mathbf{f}}(Z)^n$, we get
by Lemma \ref{orderlemma} and the assumption $N\geq3$ 
\begin{equation*}
n_0\mathbf{B}_2(1/2+v_k)+n_{1/2}\mathbf{B}_2(v_k)
=n_0\mathbf{B}_2(1/2+1/N)+n_{1/2}\mathbf{B}_2(1/N).
\end{equation*}
We then see by (\ref{nk0}) 
\begin{equation*}
v_k=\frac{1}{N}\quad\textrm{or}\quad\frac{2^{g-1}}{2^g-1}-\frac{1}{N}.
\end{equation*}
If $\displaystyle v_k=\frac{2^{g-1}}{2^g-1}-\frac{1}{N}$, then
we obtain by the facts $Nv_k\in\mathbb{Z}$ and $v_k<1/2$ that
\begin{equation*}
(2^g-1)\,|\,N\quad\textrm{and}\quad N<2(2^g-1).
\end{equation*}
Thus we must have 
\begin{equation*}
v_k=\left\{\begin{array}{ll}
 {1}/{N}&\textrm{if $N\neq 2^g-1$},\\
{1}/{N}~~\textrm{or}~~ {1}/{2}-{1}/{2N}&\textrm{if $N= 2^g-1$}.
\end{array}\right.
\end{equation*}
\item[(b)] Let $k$ be an index ($1\leq k\leq g$) such that $1/2\leq v_k<1$.
We attain by Lemma \ref{orderlemma} and the assumption $N\geq3$ 
\begin{equation*}
n_0\mathbf{B}_2(-1/2+v_k)+n_{1/2}\mathbf{B}_2(v_k)=
n_0\mathbf{B}_2(1/2+1/N)+n_{1/2}\mathbf{B}_2(1/N).
\end{equation*}
We then derive by (\ref{nk0}) 
\begin{equation*}
v_k=1-\frac{1}{N}\quad\textrm{or}\quad\frac{2^{g-1}-1}{2^g-1}+\frac{1}{N}.
\end{equation*}
Suppose $\displaystyle v_k=\frac{2^{g-1}-1}{2^g-1}+\frac{1}{N}$.
Then we achieve by the fact $Nv_k\in\mathbb{Z}$ and $v_k\geq1/2$ that
\begin{equation*}
(2^g-1)\,|\,N\quad\textrm{and}\quad N\leq2(2^g-1).
\end{equation*}
If $N=2(2^g-1)$, we have $v_k=1/2$ and so 
$\begin{bmatrix}
\langle v_k\rangle\\
\langle v_{k+g}\rangle\end{bmatrix}=\begin{bmatrix}
1/2\\
0\end{bmatrix}$ by (c).
It contradicts Lemma \ref{orderlemma}.
Hence we should have
\begin{equation*}
v_k=\left\{\begin{array}{ll}
 1-{1}/{N}&\textrm{if $N\neq 2^g-1$},\\
 1-{1}/{N}~~\textrm{or}~~{1}/{2}+{1}/{2N}&\textrm{if $N= 2^g-1$}.
\end{array}\right.
\end{equation*}
\item[(c)] Let $k$ be an index ($g+1\leq k\leq 2g$) such that
$0\leq v_k<1/2$. 
We deduce by Lemma \ref{orderlemma} 
\begin{equation*}
n_0\mathbf{B}_2(1/2+v_k)+n_{1/2}\mathbf{B}_2(v_k)
=n_0\mathbf{B}_2(1/2)+n_{1/2}\mathbf{B}_2(0).
\end{equation*}
It then follows from (\ref{nk0}) that
\begin{equation*}
v_k=0\quad\textrm{or}\quad \frac{2^{g-1}}{2^g-1}.
\end{equation*}
Since $v_k<1/2$, we must take $v_k=0$.
\item[(d)] Let $k$ be an index ($g+1\leq k\leq 2g$) such that
$1/2\leq v_k<1$. 
We see by Lemma \ref{orderlemma} that
\begin{equation*}
n_0\mathbf{B}_2(-1/2+v_k)+n_{1/2}\mathbf{B}_2(v_k)=
n_0\mathbf{B}_2(1/2)+n_{1/2}\mathbf{B}_2(0).
\end{equation*}
We then claim by (\ref{nk0}) that
\begin{equation*}
v_k=1\quad\textrm{or}\quad\frac{2^{g-1}-1}{2^g-1},
\end{equation*}
which is impossible because $1/2\leq v_k<1$. 
Therefore this case cannot happen.
\end{enumerate}
\par
Suppose that $v_i\not\equiv v_j\Mod{\mathbb{Z}}$ for some $1\leq i<j\leq g$.
Acting the matrix
\begin{equation*}
\begin{bmatrix}
I_g&O_g\\E_{1i}'-E_{1j}'&I_g
\end{bmatrix}^T\in\mathrm{Sp}_{2g}(\mathbb{Z})
\end{equation*}
on both sides of $\Theta_\mathbf{v}(Z)^n=\Theta_{(1/N)\mathbf{f}}(Z)^n$, we deduce
by Proposition \ref{bigthetaproperty} (iii) that
\begin{equation*}
\Theta_{\mathbf{u}}(Z)^n=\Theta_{(1/N)\mathbf{f}}(Z)^n,
\end{equation*}
where $\mathbf{u}=\begin{bmatrix}u_1\\\vdots\\u_{2g}\end{bmatrix}\in\mathcal{I}_N$ with 
$u_{g+1}\equiv v_i-v_j\not\equiv 0 \Mod{\mathbb{Z}}$.
On the other hand, we see from the above claim that $u_{g+1}\equiv 0\Mod{\mathbb{Z}}$,
which gives a contradiction.
Hence we have
\begin{equation*}
\mathbf{v}\equiv\left\{
\begin{array}{ll}
\pm({1}/{N})\mathbf{f}\Mod{\mathbb{Z}^{2g}}&\textrm{if $N\neq 2^g-1$},\\
\pm({1}/{N})\mathbf{f}~\textrm{or}~({1}/{2}\pm{1}/{2N})\mathbf{f}\Mod{\mathbb{Z}^{2g}}&\textrm{if $N= 2^g-1$}.
\end{array}\right.
\end{equation*}
Now, assume that $N= 2^g-1$ and $\mathbf{v}\equiv({1}/{2}\pm 1/2N)\mathbf{f}\Mod{\mathbb{Z}^{2g}}$.
If $N=3$ then $\pm1/N\equiv 1/2\mp 1/(2N)\Mod{\mathbb{Z}}$.
So we further assume that $N\neq 3$.
Acting the matrix
\begin{equation*}
\alpha=\begin{bmatrix}
I_g&O_g\\E_{1i}'+E_{1j}'&I_g
\end{bmatrix}^T\in\mathrm{Sp}_{2g}(\mathbb{Z})
\end{equation*}
for any $1\leq i<j\leq g$ on both sides of $\Theta_\mathbf{v}(Z)^n=\Theta_{(1/N)\mathbf{f}}(Z)^n$, we obtain 
\begin{equation*}
\Theta_{\alpha^T\mathbf{v}}(Z)^n=\Theta_{(1/N)\alpha^T\mathbf{f}}(Z)^n.
\end{equation*}
Here we observe that $(\alpha^T\mathbf{v})_{g+1}\equiv \pm 1/N\Mod{\mathbb{Z}}$ and 
$((1/N)\alpha^T\mathbf{f})_{g+1}\equiv 2/N\Mod{\mathbb{Z}}$.
Meanwhile, one can show by Lemma \ref{orderlemma} that
\begin{equation*}
((1/N)\alpha^T\mathbf{f})_{g+1}\equiv \pm \frac{1}{N}~~\textrm{or}~~\frac{1}{2}\pm\frac{1}{2N}\Mod{\mathbb{Z}},
\end{equation*}
which is a contradiction.
Therefore, we conclude $\mathbf{v}\equiv\pm (1/N)\mathbf{f}\Mod{\mathbb{Z}^{2g}}$.
\item[(ii)] Acting $\begin{bmatrix}
I_g&O_g\\-I_g&I_g
\end{bmatrix}^T\in\mathrm{Sp}_{2g}(\mathbb{Z})$ on both sides of $\Theta_\mathbf{v}(Z)^n=\Theta_{(1/N)\mathbf{e}}(Z)^n$,
we get
\begin{equation*}
\Theta_{\left[\begin{smallmatrix}\mathbf{v}_u\\-\mathbf{v}_u+\mathbf{v}_l
\end{smallmatrix}\right]}(Z)^n=\Theta_{(1/N)\mathbf{f}}(Z)^n.
\end{equation*}
Thus we obtain by (i) 
\begin{equation*}
\begin{bmatrix}\mathbf{v}_u\\-\mathbf{v}_u+\mathbf{v}_l\end{bmatrix}
\equiv\pm(1/N)\mathbf{f}\Mod{\mathbb{Z}^{2g}},
\end{equation*}
from which it follows that $\mathbf{v}\equiv\pm(1/N)\mathbf{e}\Mod{\mathbb{Z}^{2g}}$.
\item[(iii)] Let
\begin{equation*}
A=\left\{\begin{array}{ll}
I_g+\sum_{1\leq i\leq g,\, i\neq j} E_{ij} & \textrm{when}~1\leq j\leq g,\\
I_g+\sum_{1\leq i\leq g,\,i\neq j-g} E_{i\,j-g} & \textrm{when}~g+1\leq j\leq 2g,
\end{array}\right.
\end{equation*}
and let
\begin{equation*}
\alpha=\left\{\begin{array}{ll}
\begin{bmatrix}
A&O_g\\O_g & (A^T)^{-1}
\end{bmatrix}^T & \textrm{if}~1\leq j\leq g,\\
\begin{bmatrix}
O_g & A\\ -(A^T)^{-1} & O_g
\end{bmatrix}^T& \textrm{if}~g+1\leq j\leq 2g
\end{array}\right.
\end{equation*}
which belongs to $\mathrm{Sp}_{2g}(\mathbb{Z})$.
Acting $\alpha$ on both sides of $\Theta_\mathbf{v}(Z)^n=
\Theta_{(1/N)\mathbf{e}_j}(Z)^n$, we claim by Proposition \ref{bigthetaproperty} (iii) that
\begin{equation*}
\Theta_{\alpha^T\mathbf{v}}(Z)^n=
\Theta_{(1/N)\alpha^T\mathbf{e}_j}(Z)^n=\Theta_{(1/N)\mathbf{f}}(Z)^n.
\end{equation*}
Thus we derive (iii) by utilizing (i).
\end{enumerate}
\end{proof}

\begin{theorem}\label{mainresult}
Assume that $g\geq2$ and $N\geq 3$. 
Then, for any nonzero integer $n$, we have
\begin{align*}
\mathcal{F}_N&=
\mathcal{F}_1\left(
\Theta_{(1/N)\mathbf{e}_1}(Z)^n,
\ldots,
\Theta_{(1/N)\mathbf{e}_{2g}}(Z)^n,
\Theta_{(1/N)\mathbf{e}}(Z)^n\right),\\
\mathcal{F}^1_N(\mathbb{Q})&=
\mathcal{F}_1\left(
\Theta_{(1/N)\mathbf{e}_1}(Z)^n,\ldots,
\Theta_{(1/N)\mathbf{e}_g}(Z)^n,
\Theta_{(1/N)\mathbf{f}}(Z)^n\right),\\
\mathcal{F}_{1,N}(\mathbb{Q})&=
\mathbb{Q}\left(f(NZ),
\Theta_{(1/N)\mathbf{e}_1}(NZ)^n,\ldots,
\Theta_{(1/N)\mathbf{e}_g}(NZ)^n,
\Theta_{(1/N)\mathbf{f}}(NZ)^n~|~f(Z)\in\mathcal{F}_1\right).
\end{align*}
\end{theorem}
\begin{proof}
Let $\{f_\mathbf{v}(Z)\}_{\mathbf{v}\in\mathcal{I}_N}$ be a Siegel family. 
Note that in Propositions \ref{generators1} and
\ref{generators2} we only require the family $\{f_\mathbf{v}(Z)\}_{\mathbf{v}\in\mathcal{I}_N}$ 
to satisfy the property
\begin{equation*}
f_\mathbf{v}(Z)=f_{\mathbf{v}'}(Z)~\Longleftrightarrow~
\mathbf{v}\equiv\pm\mathbf{v}'\Mod{\mathbb{Z}^{2g}}\quad
\end{equation*}
when $\mathbf{v}'=\mathbf{e}$, $\mathbf{f}$ or $\mathbf{e}_j$ ($1\leq j\leq 2g$).
Hence the theorem follows from Lemma \ref{weaklemma}. 
\end{proof}

\begin{example}
In particular, let $g=2$. As is well known, we have
\begin{equation*}
\mathcal{F}_1=\mathbb{Q}\left(\frac{E_4(Z)E_6(Z)}{E_{10}(Z)},
\frac{E_6(Z)^2}{E_{12}(Z)},\frac{E_4(Z)^5}{E_{10}(Z)^2}\right),
\end{equation*}
where $E_{2k}(Z)$ is the Siegel Eisenstein series of weight $2k$
(\cite[Theorem 3]{Igusa62} or \cite[\S 10 Proposition 3]{Klingen}).
Thus, if $N\geq 3$, then we deduce by Theorem
\ref{mainresult} that
\begin{eqnarray*}
\mathcal{F}_N&=&\mathbb{Q}\left(\frac{E_4(Z)E_6(Z)}{E_{10}(Z)},
\frac{E_6(Z)^2}{E_{12}(Z)},\frac{E_4(Z)^5}{E_{10}(Z)^2},
\Theta_{\left[\begin{smallmatrix}
1/N\\0\\0\\0
\end{smallmatrix}\right]}(Z),
\Theta_{\left[\begin{smallmatrix}
0\\1/N\\0\\0
\end{smallmatrix}\right]}(Z),
\Theta_{\left[\begin{smallmatrix}
0\\0\\1/N\\0
\end{smallmatrix}\right]}(Z),\right.\\
&&\left.
\qquad\Theta_{\left[\begin{smallmatrix}
0\\0\\0\\1/N
\end{smallmatrix}\right]}(Z),
\Theta_{\left[\begin{smallmatrix}
1/N\\1/N\\1/N\\1/N
\end{smallmatrix}\right]}(Z)\right),\\
\mathcal{F}^1_N(\mathbb{Q})&=&\mathbb{Q}\left(\frac{E_4(Z)E_6(Z)}{E_{10}(Z)},
\frac{E_6(Z)^2}{E_{12}(Z)},\frac{E_4(Z)^5}{E_{10}(Z)^2},
\Theta_{\left[\begin{smallmatrix}
1/N\\0\\0\\0
\end{smallmatrix}\right]}(Z),
\Theta_{\left[\begin{smallmatrix}
0\\1/N\\0\\0
\end{smallmatrix}\right]}(Z),
\Theta_{\left[\begin{smallmatrix}
1/N\\1/N\\0\\0
\end{smallmatrix}\right]}(Z)\right),\\
\mathcal{F}_{1,N}(\mathbb{Q})&=&\mathbb{Q}\left(\frac{E_4(NZ)E_6(NZ)}{E_{10}(NZ)},
\frac{E_6(NZ)^2}{E_{12}(NZ)},\frac{E_4(NZ)^5}{E_{10}(NZ)^2},
\Theta_{\left[\begin{smallmatrix}
1/N\\0\\0\\0
\end{smallmatrix}\right]}(NZ),
\Theta_{\left[\begin{smallmatrix}
0\\1/N\\0\\0
\end{smallmatrix}\right]}(NZ),
\Theta_{\left[\begin{smallmatrix}
1/N\\1/N\\0\\0
\end{smallmatrix}\right]}(NZ)\right).
\end{eqnarray*}

\end{example}

\bibliographystyle{amsplain}

\address{% the first author
Department of Mathematical Sciences \\
KAIST \\
Daejeon 34141\\
Republic of Korea} {jkkoo@math.kaist.ac.kr}
\address{% the corresponding author
Department of Mathematics\\
Hankuk University of Foreign Studies\\
Yongin-si, Gyeonggi-do 17035\\
Republic of Korea} {dhshin@hufs.ac.kr}
\address{
Department of Mathematical Sciences \\
KAIST \\
Daejeon 34141\\
Republic of Korea} {math\_dsyoon@kaist.ac.kr}

\end{document}